\pdfoutput=1

\documentclass[journal,twocolumn]{IEEEtran}
\usepackage[square,sort&compress,comma,numbers]{natbib}
\usepackage{kz}
\usepackage{slashbox}
\usepackage{enumerate}
\usepackage{booktabs}
\newcommand{\ra}[1]{\renewcommand{\arraystretch}{#1}}
\usepackage{lipsum}

\ifCLASSINFOpdf
 \else
\fi

\begin{document}
\title{Spherical Cap Packing Asymptotics \\ and Rank-Extreme Detection}

\author{Kai Zhang
\thanks{Kai Zhang is with the Department of Statistics and Operations Research, University of North Carolina, Chapel Hill (e-mail: zhangk@email.unc.edu)}}
\markboth{IEEE TRANSACTIONS ON INFORMATION THEORY}%
{Zhang: Spherical Cap Packing Asymptotics}
\maketitle

\begin{abstract}
We study the spherical cap packing problem with a probabilistic approach. Such probabilistic considerations result in an asymptotic sharp universal uniform bound on the maximal inner product between any set of unit vectors and a stochastically independent uniformly distributed unit vector. When the set of unit vectors are themselves independently uniformly distributed, we further develop the extreme value distribution limit of the maximal inner product, which characterizes its uncertainty around the bound.

As applications of the above asymptotic results, we derive (1) an asymptotic sharp universal uniform bound on the maximal spurious correlation, as well as its uniform convergence in distribution when the explanatory variables are independently Gaussian distributed; and (2) an asymptotic sharp universal bound on the maximum norm of a low-rank elliptically distributed vector, as well as related limiting distributions. With these results, we develop a fast detection method for a low-rank structure in high-dimensional Gaussian data without using the spectrum information.
\end{abstract}


\IEEEpeerreviewmaketitle

\section{Introduction}\label{sec:intro}
In modern data analysis, datasets often contain a large number of variables with complicated dependence structures. This situation is especially common in important problems such as the relationship between genetics and cancer, the association between brain connectivity and cognitive states, the effect of social media on consumers' confidence, etc. Hundreds of research papers on analyzing such dependence have been published in top journals and conferences proceedings. For a comprehensive review of these challenges and past studies, see \cite{Fan2014}.

One of the most important measures on the dependence between variables is the correlation coefficient, which describes their linear dependence. In the new paradigm described above, understanding the correlation and the behavior of correlated variables is a crucial problem and prompts data scientists to develop new theories and methods. Among the important challenges of a large number of variables on the correlation, we focus particularly on the following two questions:

\begin{itemize}
  \item {\bf The maximal spurious sample correlation in high dimensions.} The Pearson's sample correlation coefficient between two random variables $X$ and $Y$ based on $n$ observations can be written as
      \begin{equation}
        \hat{C}(X,Y) = {\sum_{i=1}^n (X_i-\bar{X})(Y_i-\bar{Y})\over \sqrt{\sum_{i=1}^n (X_i-\bar{X})^2}\sqrt{\sum_{i=1}^n(Y_i-\bar{Y})^2}}
      \end{equation}
      where $X_i$'s and $Y_i$'s are the $n$ independent and identically distributed (i.i.d.) observations of $X$ and $Y$ respectively, and $\bar{X}$ and $\bar{Y}$ are the sample means of $X$ and $Y$ respectively. The sample correlation coefficient possesses important statistical properties and was carefully studied in the classical case when the number of variables is small compared to the number of observations. However, the situation has dramatically changed in the new high-dimensional paradigm \citep{Johnstone13112009,Fan2014} as the large number of variables in the data leads to the failure of many conventional statistical methods. For sample correlations, one of the most important challenges is that when the number of explanatory variables, $p$, in the data is high, simply by chance, some explanatory variable will appear to be highly correlated with the response variable even if they are all scientifically irrelevant \citep{fan2011sparse,Fan2012}. Failure to recognize such spurious correlations can lead to false scientific discoveries and serious consequences. Thus, it is important to understand the magnitude and distribution of the maximal spurious correlation to help distinguish signals from noise in a large-$p$ situation.
  \item {\bf Detection of low-rank correlation structure.} Detecting a low-rank structure in a high-dimensional dataset is of great interest in many scientific areas such as signal processing, chemometrics, and econometrics. Current rank estimation methods are mostly developed under the factor model and are based on the principal component analysis (PCA) \citep{Markovsky2012, yang365290,rabideau1996fast, kavcic506625,real1999two, shi2007adaptive,badeau2008fast, Bartelmaos4479581,doukopoulos2008fast, Kritchman200819,Kritchman2009, johnstone2009consistency,perry5447660,berthet2013,donoho2013optimal,cai2014,choi2014selecting}, where we look for the ``cut-off'' among singular values of the covariance matrix when they drop to nearly 0. These methods also usually assume a large sample size. However, in practice often a large number of variables are observed while the sample size is limited. In particular, PCA based methods will fail when the number of observations is less than the rank. Moreover, although we may get low-rank solutions to many problems, more detailed inference on the rank as a parameter is not very clear. Probabilistic statements on the rank, such as confidence intervals and tests, would provide useful information about the accuracy of these solutions. The computation complexity of the matrix calculations can be an additional issue in practice. In summary, it is desirable to have a fast detection and inference method of a low-rank structure in high dimensions from a small sample.
\end{itemize}
Our study of the above two problems starts with the following question: Suppose $p$ points are placed on the unit sphere $\cS^{n-1}$ in $\RR^{n}.$ If we now generate a new point on $\cS^{n-1}$ according to the uniform distribution over the sphere, how far will it be away from these existing $p$ points?

Intuitively, this minimal distance between the new point and the existing $p$ points should depend on $n$ and $p$, in a manner that it is decreasing in $p$ and increasing in $n.$ Yet, no matter how these existing $p$ points are located, this new point cannot get arbitrarily close to the existing $p$ points due to randomness. In other words, for any $n$ and $p,$ there is an intrinsic lower bound on this distance that the new point can get closer to the existing points only with very small probability.

Studies of this intrinsic lower bound in the above question have a long history under the notion of spherical cap packing, and this question has been one of the most fundamental questions in mathematics \citep{Rankin1955, thompson1983,Conway1987}. In fact, this question is closely related to the 18th question on the famous list from David Hilbert \citep{hilbert1902}. This question is also a very important problem in information theory and has been studied in coding, beamforming, quantization, and many other areas \citep{wyner1967, Barg1027776, mukkavilli2003beamforming, henkel2005sphere,koetter2008coding,dai2008quantization,pitaval2011beamforming,Dalai6612702}.

Besides the importance in mathematics and information theory, this question is closely connected to the two problems that we propose to investigate. For instance, the sample correlation between $X$ and $Y$ can be written as the inner product
\begin{equation}\label{eq.corr.ind}
  \hat{C}(X,Y)=\langle {\bX-\bar{X}\one_n \over \|\bX-\bar{X}\one_n \|_2},{\bY-\bar{Y}\one_n  \over \|\bY-\bar{Y} \one_n \|_2} \rangle,
\end{equation}
where $\bX=(X_1,\ldots,X_n),$ $\bY=(Y_1,\ldots,Y_n),$ and $\one_n$ is the vector in $\RR^n$ with all ones. In general, if we observe $n$ i.i.d. samples from the joint distribution of $(X_1,\ldots,X_p,Y)$, the sample correlations between $X_j$'s and $Y$ can be regarded as inner products in $\RR^n$ between the $p$ unit vectors corresponding to $X_j$'s and another unit vector corresponding to $Y.$ Note that these unit vectors are all orthogonal to the vector $\one_n$ due to the centering process. Thus, they lie on an ``equator'' of the unit sphere $\cS^{n-1}$ in $\RR^n,$ which is in turn equivalent to $\cS^{n-2}.$ Through this connection, the problem about the maximal spurious correlation is equivalent to the packing of the inner products, and existing methods and results from the packing literature can be borrowed to analyze this problem. In this paper, we particularly focus on probabilistic statements about such packing problems.

An important advantage of this packing perspective is a view of data that is free of an increasing $p$. Suppose we view the data as $n$ points in a $p$-dimensional space, then if $p$ exceeds $n,$ all the $n$ points will lie on a low-dimensional hyperplane in $\RR^p.$ This degeneracy forces us to change the methodology towards statistical problems, i.e., changing from the classical statistical methods to recent high-dimensional methods \citep{htf2009,buhlmann2011}. However, if we view the data as $p$ vectors in $\RR^n$, then we will never have such a degeneracy problem. No matter how large $p$ is, a packing problem is always a well-defined packing problem. Neither the theory nor the methodology needs to be changed due to an increase in $p$. Thus, with the packing perspective, theory and methodology can be set free from the restriction of an increasing $p$.

We summarize below our results on the asymptotic theories of the maximal inner products and spurious correlations. One major advantage of the packing approach is that instead of usual iterative asymptotic results which set $p=p(n)$ and let $n \rightarrow \infty,$ our convergence results are uniform in $n$, which leads to double limits in both $n$ and $p$.
\begin{itemize}
  \item We characterize the largest magnitude of independent inner products (or spurious correlations) through an asymptotic bound. This bound is universal in the sense that it holds for arbitrary distributions of $\bL_j$'s (or that of $X_j$'s). This bound is uniform in the sense that it holds asymptotically in $p$ but is uniform over $n.$ This bound is sharp in the sense that it can be attained, especially when the unit vectors $\bL_j$'s are i.i.d. uniform (or when $X_j$'s are independently Gaussian). Thus, in an analogy, this bound is to the distribution of independent inner products (or to that of spurious correlations) as the fundamental bound $\sqrt{2 \log p}$ is to the $p$-dimensional Gaussian distribution \citep{leadbetter1983}. We refer this bound as the Sharp Asymptotic Bound for indEpendent inner pRoducts (or spuRious corrElations), abbreviated as the SABER (or SABRE).
  \item In the special important case when the set of unit vectors are i.i.d. uniformly distributed (or when $X_j$'s are independently Gaussian distributed), we show the sharpness of the SABER (or SABRE) and describe a smooth phase transition phenomenon of them according to the limit of ${\log p \over n}$. Furthermore, we develop the limiting distribution by combing the packing approach with extreme value theory in statistics \citep{leadbetter1983,Haan2006}. The extreme value theory results accurately characterize the deviation from the observed maximal magnitude of independent inner products (or that of spurious correlations) to the SABER (or SABRE). One important feature of these results is that they are not only finite sample results but also are uniform-$n$-large-$p$ asymptotics that are widely applicable in the high-dimensional paradigm.
\end{itemize}

The spherical cap packing asymptotics can be also applied to the problem of the detection of a low-rank linear dependency. For this problem, we observe that the largest magnitude among $p$ standard elliptical variables is closely related to the rank $d$ of their correlation matrix. This is seen by decomposing elliptically distributed random vectors into the products of common Euclidean norms and inner products of unit vectors in $\RR^d$, thus reducing the problem to one of spherical cap packing. As a consequence, the previous asymptotics can be applied here. We thus obtained a universal sharp asymptotic bound for the maximal magnitude of a degenerate elliptical distribution, as well as its limiting distribution when the unit vectors in the decomposition are i.i.d. uniform. Although many asymptotic bounds and limiting distributions on full rank maxima are well developed under different situations (see \citep{leadbetter1983,adler1990introduction,Haan2006} for reviews of the extensive existing literature), we are not able to find similar theory in literature on low-rank maxima from elliptical distributions, not even in the special case of Gaussian distribution. We refer the connection we found between the maximal magnitude and the rank as the rank-extreme (ReX) association.

Based on the asymptotic results on the degenerate elliptical distributions, we show that one can make statistical inference on a low-rank through the distributions of the extreme value as a statistic. One feature of this procedure is that it does not require the spectrum information from PCA. Thus, the new method works when $n<d$, when PCA based methods fail to work. It is also computationally fast since no matrix multiplication is needed in the algorithm. These advantages allow a fast detection of a low-dimensional correlation structure in high-dimensional data.

\subsection{Related Work}
We are not able to find similar probabilistic statements on uniform-$n$-large-$p$ asymptotics. The following statistical papers are related to the study on the maximal spurious correlation.
\begin{itemize}
  \item In \cite{Fan2012}, the authors obtain a result on the order of the maximal spurious correlations in the regime that ${\log p \over n} \rightarrow 0.$ Through the packing approach, we derive the explicit limiting distribution of the extreme spurious correlations for entire scope of $n$ and $p.$
  \item In \cite{hero2012}, the authors develop a threshold for marginal correlation screening with large $p$ and small $n.$ The threshold appears in a similar form as the SABRE. We note two major differences between the results: (1) The results in \cite{hero2012} focus on the regime when ${\log p\over n} \rightarrow \infty$ (i.e., when the threshold converges to 1), while our asymptotic results cover the entire scope of $n$ and $p$, and the SABRE is shown to be valid from 0 to 1; (2) we derive the explicit limiting distribution of the maximal spurious correlation in the most important case when the variables are i.i.d. Gaussian.
  \item In \cite{cj2011,cj2012,cfj2013}, the minimal pairwise angles between i.i.d. uniformly random points on spheres are considered. A similar phase transition is described, and results on the limiting distribution are developed. We note two major differences between their results and ours: (1) Due to different motivations of the research, we focus on the marginal correlation between one response variable and $p$ explanatory variables. We also develop a universal uniform bound for marginal correlations. (2) The extreme limiting distributions in their papers are stated separately according to if the limit of ${\log p \over n}$ is $0,$ a proper constant, or $\infty.$ From the packing perspective, we are able to state the convergence in a uniform manner with standardizing constants that are adaptive in $n$ and $p.$ Since in real data, the limit of ${\log p \over n}$ is usually not known, this uniform convergence with adaptive standardizing constants makes the result easy to apply in practice.
  \item During the review process of this paper, we noticed the results in \cite{fan2015discoveries} which focus on the coupling and bootstrap approximations of the maximal spurious correlation when ${\log^7 p \over n} \rightarrow 0$. Again our different focus is on explicit limiting distributions with adaptive standardizing constants from the packing perspective.
\end{itemize}
We are not able to find existing literature on the rank-extreme association. To evaluate the performance of our low-rank detection method, we compare our method with the algorithm in \citep{Kritchman200819} which studies a similar problem. During the review process of the paper, we also noticed recent work by \cite{choi2014selecting}. The most important difference from these papers is that they focus on the case when $n$ and $p$ are comparable and both large, while we consider the case when $n$ is small and $p$ is large.

\subsection{Outline of the Paper}
In Section~\ref{sec:spur}, we derive the asymptotic bound on the spherical packing problem, as well as that of the maximal spurious correlation and the related extreme value distributions. In Section~\ref{sec:rex}, we describe the rank-extreme association of elliptically distributed vectors. In Section~\ref{sec:rex.detect} we develop a fast detection method of a low-rank by using the rank-extreme association reversely. In Section~\ref{sec:simu}, we study the performance of the detection method through simulations. We conclude and discuss future work in Section~\ref{sec:disc}.

\section{Asymptotic Theory of the Spherical Cap Packing Problem}\label{sec:spur}
\subsection{The Sharp Asymptotic Bound for independent inner products (SABER) and Spurious Correlations (SABRE)}\label{subsec:spur.uub}
We first observe that as described in \citep{muirhead1982}, when $\bU$ is uniformly distributed over $\cS^{n-1},$ $|\langle \bL, \bU\rangle|^2 \sim Beta\big({1 \over 2},{ n-1 \over 2}\big), \forall \bL \in \cS^{n-1}$. By borrowing strength from the packing literature \citep{Rankin1955,wyner1967} on the total area of non-overlap spherical caps on $\cS^{n-2}$, we develop the following theorem on a sharp asymptotic bound for independent inner products.

\begin{theorem}\label{thm.saber}
{\bf Sharp Asymptotic Bound for Independent Inner Products (SABER).}\\
For arbitrary deterministic unit vectors $\bL_1,\ldots,\bL_p$ and a uniformly distributed unit vector $\bU$ over $\cS^{n-1}$, the random variable $M_{p,n}=\max_{1 \le j \le p} |\langle \bL_j, \bU\rangle|$ satisfies that $\forall \delta >0$,
{\small
\begin{equation}\label{ineq.sabre.finite}
\begin{split}
   &  \Pb\bigg(M_{p,n} > \sqrt{(1+\delta)(1-p^{-2/(n-1)})}\bigg) \\
\le & {\sqrt{2}p^{1/(n-1)} \exp\bigg(-{1 \over 2}\delta(n-1)(p^{2/(n-1)}-1)\bigg) \over \sqrt{\pi (1+\delta)(n-1)(p^{2/(n-1)}-1)}}.
\end{split}
\end{equation}
Therefore, $\forall \delta>0$,  as $p \rightarrow \infty,$
\begin{equation}
    \sup_{n \ge 2}  \Pb\bigg(M_{p,n} > \sqrt{(1+\delta)(1-p^{-2/(n-1)})}\bigg) \rightarrow 0.
\end{equation}
In particular, if $n \rightarrow \infty,$ then we have the double limit
\begin{equation}\label{ineq.sabre.large}
 \lim_{p,n \rightarrow \infty} \Pb\bigg(M_{p,n} \le \sqrt{1-p^{-2/(n-1)}} \bigg) = 1.
\end{equation}
}
\end{theorem}

Theorem~\ref{thm.saber} provides an explicit answer to the question at the beginning of Section~\ref{sec:intro} with a probabilistic statement: No matter how $\bL_j$'s are located on the unit sphere, the magnitude of the inner products (or cosines of the angle) between these $p$ points and a uniformly random point cannot exceed $\sqrt{1-p^{-2/(n-1)}}$ with high probability for large $p.$ This upper bound on the inner products is equivalent to a lower bound on the minimal angle between the new random point to the existing $p$ points.

The SABER possesses the following important properties:
\begin{enumerate}
  \item This bound is universal in the sense that it holds for any configuration of $\bL_j$'s.
  \item This bound is uniform in the sense that it holds uniformity for $n \ge 2$.
  \item This bound is sharp in the sense that it can be attained for some configuration of $\bL_j$'s, especially when $\bL_j$'s are i.i.d. uniformly distributed, as will be discussed in Section~\ref{subsec:spur.iid}.
\end{enumerate}
Thus, in an analogy, the SABER $\sqrt{1-p^{-2/(n-1)}}$ is to the distributions of the independent inner products as the fundamental bound $\sqrt{2 \log p}$ is to the $p$-dimensional Gaussian distribution.

A technical note here is that when $n$ is finite, the fraction ${2 \over n-1}$ in the exponent of $p$ can be replaced by ${2 \over n-n_1}$ with any fixed integer $0< n_1<n.$ This change would not alter the asymptotic result in $p$ due to a uniform convergence in the proof. The number $n_1$ only has an effect when the dimension $n$ is finite. For example, see \cite{hero2012} for a similar but different bound when $n$ is fixed. We focus on the bound $\sqrt{1- p^{-2/(n-1)}}$ due to its connection to the $Beta\big({1 \over 2},{ n-1 \over 2}\big)$ distribution. When $n \rightarrow \infty,$ all these bounds are equivalent.

Another technical note is that although Theorem~\ref{thm.saber} is for a deterministic set of $\bL_j$'s, we note here that this set of unit vectors can be random as well. As long as $\bL_j$'s are stochastically independent of $\bU,$ Theorem~\ref{thm.saber} can be applied to random $\bL_j$'s by a conditioning argument on any realization of $\bL_j$'s.

Figure~\ref{figure:sabre} illustrates the SABER $\sqrt{1-p^{-2/(n-1)}}$ in Theorem~\ref{thm.saber} as a function of $n$ and $p$. It can be seen that the SABER has a range of $(0,1)$ as an increasing function in $p$ and a decreasing function in $n.$

\begin{figure}[hhhh]
\begin{center}
\includegraphics[width=3.2in]{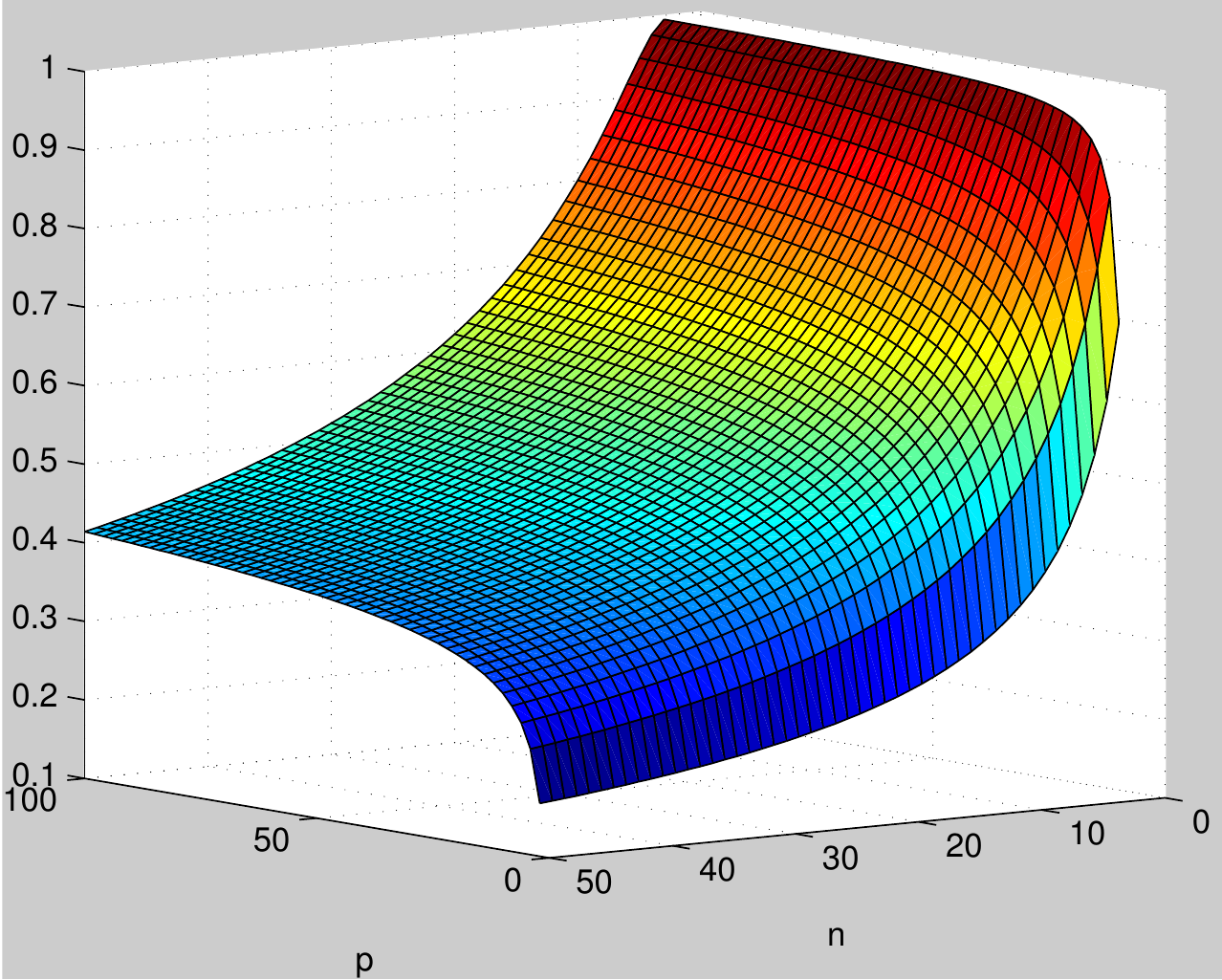}
\end{center}
\caption{The SABER $\sqrt{1-p^{-2/(n-1)}}$ for $3 \le n \le 50$ and $2 \le p \le 100$. The SABER ranges from 0 to 1. The regions of the same color represent the smooth phase transition curves ${\log p \over n} \approx \beta$ for $\beta>0$ as described in Section~\ref{subsec:spur.iid}. }\label{figure:sabre}
\end{figure}

Due to the connection between sample correlations and the inner products \eqref{eq.corr.ind}, this bound is immediately applicable to spurious correlations. Suppose $\bY=(Y_1,\ldots,Y_n)$ records $n$ i.i.d. samples of a Gaussian variable $Y,$ then it is well-known (see \citep{muirhead1982}) that ${\bY-\bar{Y}\one_n \over \|\bY-\bar{Y}\one_n\|_2}$ is a uniformly distributed unit vector over $\cS^{n-2}$. Thus, we have the following bound on the maximal spurious correlation.

\begin{corollary}\label{cor.sabre}
{\bf Sharp Asymptotic Bound for Spurious Correlations (SABRE).}\\
Suppose we observe $n$ i.i.d. samples of arbitrary random variables $X_1,\ldots,X_p$ and a Gaussian variable $Y$ that is independent of $X_j$'s. The maximal absolute sample correlation
$
  M_{XY}=\max_{1 \le j \le p} |\hat{C}(X_j,Y)|
$
satisfies that $\forall \delta>0$, as $p \rightarrow \infty,$
\begin{equation}
    \sup_{n \ge 3}  \Pb\bigg(M_{XY} > \sqrt{(1+\delta)(1-p^{-2/(n-2)})}\bigg) \rightarrow 0.
  \end{equation}
In particular, if $n \rightarrow \infty,$ then we have the double limit
\begin{equation}
 \lim_{p,n \rightarrow \infty} \Pb\bigg(M_{XY}  \le \sqrt{1-p^{-2/(n-2)}} \bigg) = 1.
\end{equation}
\end{corollary}

Similarly as the interpretation for the SABER, the implication of the SABRE is as follows: Uniformly for $n \ge 3,$ no matter how the $p$ variables $X_1,\ldots,X_p$ are distributed, the magnitude of the sample correlations between $X_j$'s and a Gaussian $Y$ cannot exceed the SABRE with high probability for large $p.$ Note here that in practice, the requirement of Gaussianity of $Y$ can be easily relaxed through a transformation of distributions. Since the SABRE is universal, uniform, and sharp as the SABER, this bound provides a way to distinguish true signals from spurious correlations. We shall investigate this application in future work.

\subsection{Limiting Distributions in the i.i.d. Case}\label{subsec:spur.iid}
In this section, we describe the asymptotics of the maximal inner product when $\bL_j$'s are i.i.d. uniformly distributed and the asymptotics of spurious correlations when $X_j$'s are independently Gaussian distributed. We first observe that when $\bL_j$'s are i.i.d. uniformly unit vectors over $\cS^{n-1},$ then for any random unit vector $\bU$ that is independent of $\bL_j$'s, we have the following two properties about the inner products $\langle \bL_j, \bU\rangle | \bU$:
\begin{enumerate}
  \item Conditioning on $\bU,$ the variables $\langle \bL_j, \bU\rangle | \bU$'s are independent since $\bL_j$'s are independent;
  \item For each $j,$ the variable $|\langle \bL_j, \bU\rangle|^2 | \bU$ is distributed as $Beta\big({1 \over 2},{n-1 \over 2} \big)$. Since this conditional distribution does not depend on $\bU$, it implies that unconditionally $|\langle \bL_j, \bU\rangle|^2 $ is stochastically independent of $\bU.$
\end{enumerate}
From these two properties, we conclude that unconditionally, $|\langle \bL_j, \bU\rangle|^2$'s are i.i.d. $Beta\big({1 \over 2},{n-1 \over 2} \big)$ distributed. We thus show the sharpness of the SABER and SABRE by studying the maximum of i.i.d. $Beta({1 \over 2},{n-1 \over 2})$ variables.
\begin{theorem}\label{thm.sabre.sharp}
\ \begin{enumerate}
  \item ({\bf Sharpness of SABER})\\
Suppose $\bL_j$'s are i.i.d. uniformly distributed over the $(n-1)$-sphere $\cS^{n-1}, $ then for arbitrary random unit vector $\bU$ that is independent of $\bL_j$'s, uniformly for all $n \ge 2$, as $p \rightarrow \infty,$ the random variable $M_{p,n}=\max_j |\langle \bL_j, \bU\rangle|$ has the following convergence:
\begin{equation}
      M_{p,n}/\sqrt{1-p^{-2/(n-1)}} \stackrel{prob.}{\longrightarrow} 1,
\end{equation}
i.e.,
$\forall \delta>0,$ as $p \rightarrow \infty,$
\begin{equation}
 \sup_{n \ge 2} \Pb( |M_{p,n}^2 /(1-p^{-2/(n-1)})-1| > \delta )\longrightarrow 0.
\end{equation}
\item ({\bf Sharpness of SABRE})\\
Similarly, suppose we observe $n$ i.i.d. samples of independent Gaussian variables $X_1,\ldots,X_p$ and an arbitrarily distributed random variable $Y$ that is independent of $X_j$'s. Consider the maximal absolute sample correlation $  M_{XY}=\max_{1 \le j \le p} |\hat{C}(X_j,Y)|.$ Uniformly for all $n \ge 3$, as $p \rightarrow \infty,$ we have
\begin{equation}
  M_{XY} /\sqrt{1-p^{-2/(n-2)}} \stackrel{prob.}{\longrightarrow} 1.
\end{equation}
\end{enumerate}
\end{theorem}

Theorem~\ref{thm.sabre.sharp} shows the sharpness of the SABER and the SABRE. It further describes a smooth phase transition of $M_{p,n}$ (also $M_{XY}$) depending on the limit of ${\log p \over n}:$
  \begin{enumerate}[(i)]
    \item If $\lim_{p \rightarrow \infty} \log{p}/n =\infty$, then $    M_{p,n} \stackrel{prob.}{\longrightarrow} 1$  and $M_{p,n}/\sqrt{1-p^{-2/(n-1)}} \stackrel{prob.}{\longrightarrow} 1.$
    \item If $\lim_{p \rightarrow \infty}  \log{p}/n = \beta$ for fixed $0 <\beta < \infty$, then
    $
      M_{p,n} \stackrel{prob.}{\longrightarrow} \sqrt{1-e^{-2\beta}}.
    $
    \item If $\lim_{p \rightarrow \infty} \log{p}/n =0$, then
     $
      M_{p,n} \stackrel{prob.}{\longrightarrow} 0
    $
    and
    $
      M_{p,n} / \sqrt{2 \log p/n}\stackrel{prob.}{\longrightarrow} 1.
    $
  \end{enumerate}

Note in particular that when $\lim_{p \rightarrow \infty} \log{p}/n =0$, the SABRE satisfies
\begin{equation}
\begin{split}
& \sqrt{1-p^{-2/(n-2)}} \\
=& \sqrt{1-e^{-2 \log p /(n-2)}} \\
\sim & \sqrt{1-(1-2 \log p/n )} \\
=& \sqrt{2 \log p /n}.
\end{split}
\end{equation}
The rate $\sqrt{2 \log p /n}$ has appeared in hundreds of books and papers and is very-well known in high-dimensional statistics literature \citep{buhlmann2011}. However, it is just a special case of the general rate $\sqrt{1-p^{-2/(n-2)}}$, which is obtained through the packing perspective. This fact demonstrates the power of this packing approach. In Figure~\ref{figure:sabre}, the smooth phase transition curves ${\log p \over n} \approx \beta$ are represented as regions of the same color.

Below are some geometric intuitions on why the phase transition depends on the limit of ${\log p \over n}$: Note that the number of orthants in $\RR^{n}$ is $2^{n}$ and is growing exponentially in $n.$ Therefore, if the growth of $p$ is faster than the exponential rate in $n,$ then the $p$ unit vectors on $\cS^{n-1}$ would be so ``dense'' that they would cover the sphere, making the magnitude of the maximal inner product converging to 1; if the growth of $p$ is exponential in $n,$ then there would be a constant number (depending on the limit of ${\log p \over n}$) of points in each orthant, so that the new random point would stay around some proper angle to the existing points; if the growth of $p$ is slower than the exponential rate, then many orthants would be empty of points asymptotically, thus the new random point can be almost orthogonal to the existing points.

When $\bL_j$'s are i.i.d. uniformly distributed or when $X_j$'s are independently Gaussian, by combining the results in packing literature \citep{Rankin1955,wyner1967} and classical extreme value theory \citep{leadbetter1983,Haan2006}, we further develop the following uniform convergence in distribution of the corresponding maxima.

\begin{theorem}\label{thm.sabre.iid}
\ \begin{enumerate}
  \item ({\bf Limiting Distribution of the Maximal Independent Inner Product})\\
  Suppose $\bL_j$'s are i.i.d. uniformly unit vectors over $\cS^{n-1}.$ For arbitrary random unit vector $\bU$ that is independent of $\bL_j$'s, consider $M_{p,n} = \max_{1 \le j \le p} |\langle \bL_j, \bU\rangle|$.
Let
$$a_{p,n} = 1- p^{-2/(n-1)} c_{p,n},  ~~~ ~~~ b_{p,n} = {2 \over n-1} p^{-2/(n-1)} c_{p,n},$$
where $c_{p,n}=\big({n-1 \over 2}B\big({1 \over 2},{n-1 \over 2}\big)\sqrt{1-p^{-2/(n-1)}}\big)^{2/(n-1)}$ is a correction factor with $B(s,t)$ being the Beta function. Then for any fixed $x$, as $p \rightarrow \infty,$
{\small
\begin{equation}
\begin{split}
  & \sup_{n \ge 2} \left|\Pb\bigg({M_{p,n}^2 - a_{p,n} \over b_{p,n}} < x\bigg)-I\bigg(x>{n-1 \over 2}\bigg) \right.\\
  &\left. - \exp\bigg(-\bigg(1-{2 \over n-1}x\bigg)^{(n-1)/2}\bigg)I\bigg(x \le {n-1 \over 2}\bigg)\right| \rightarrow 0.
\end{split}
\end{equation}
}
In particular, if $n \rightarrow \infty$ and $p \rightarrow \infty,$ then for any fixed $x$, we have the double limit
\begin{equation}
  \Pb\bigg({M_{p,n}^2 - a_{p,n} \over b_{p,n}} < x\bigg) \rightarrow \exp\big(-e^{-x}\big).
\end{equation}

  \item ({\bf Limiting Distribution of the Maximal Spurious Correlation})\\
  Similarly, suppose we observe $n$ i.i.d. samples of independent Gaussian variables $X_1,\ldots,X_p$ and an arbitrarily distributed random variable $Y$ that is independent of $X_j$'s. Consider the maximal absolute sample correlation $   M_{XY}=\max_{1 \le j \le p} |\hat{C}(X_j,Y)|.$ Then for any fixed $x$, as $p \rightarrow \infty,$
{\small
\begin{equation}
\begin{split}
  & \sup_{n \ge 3} \left|\Pb\bigg({M_{XY}^2 - a_{p,n-1} \over b_{p,n-1}} < x\bigg) -I\bigg(x>{n-2 \over 2}\bigg) \right.\\
  & \left. -\exp\bigg(-\bigg(1-{2 \over n-2}x\bigg)^{(n-2)/2}\bigg)I\bigg(x \le {n-2 \over 2}\bigg)\right| \rightarrow 0.
\end{split}
\end{equation}
}
In particular, if $n \rightarrow \infty$ and $p \rightarrow \infty,$ then for any fixed $x$, we have the double limit
\begin{equation}
  \Pb\bigg({M_{XY}^2 - a_{p,n-1} \over b_{p,n-1}} < x\bigg) \rightarrow \exp\big(-e^{-x}\big).
\end{equation}
\end{enumerate}
\end{theorem}

Theorem~\ref{thm.sabre.iid} characterizes the uncertainty of the maximal independent inner product and the maximal spurious correlation from the SABER and SABRE respectively. This result possesses the following desirable properties for practice: (1) The convergence of $M_{p,n}$ ($M_{XY}$) is uniform for $n \ge 2$ ($n \ge 3$) and is applicable provided the dataset contains two (three) observations. This uniformity over $n$ is due to the packing perspective. (2) The convergence is arbitrary for any distribution of $Y$. This arbitrariness results from the invariance property of the uniform distribution over the sphere. (3) The convergence is adaptive to the number of variables $p$: Despite the phase transition phenomenon, the normalizing constants $a_{p,n}$ and $b_{p,n}$ adaptively adjust themselves for different $n$ and $p$ to guarantee a good approximation to a proper limiting distribution. (4) Instead of the ``curse of dimensionality,'' the convergence is a ``blessing of dimensionality'': The larger $p$ is, the better the approximation is. These properties make the result widely applicable in the high-dimension-and-low-sample size situations.

We also remark here that for statistical applications, although in principle the empirical distribution of $M_{XY}$ can be simulated based on the Gaussian assumptions, in a large-$p$ situation, for example $p=10^{10}$, such simulation can incur extremely high time and computation cost. On the other hand, these quantiles can be easily obtained through the formulas of $a_{p,n}$ and $b_{p,n}$ for an arbitrary large $p$. Indeed, in modern data analysis, it is more and more often to encounter datasets with a number of variables in millions, billions, or even larger scales \cite{Hall2005}. The uniform-$n$-large-$p$ type asymptotics presented in this paper can be especially useful in these situations.

\section{Rank-Extreme Association of Degenerate Elliptical Vectors}\label{sec:rex}
\subsection{Rank-Extreme Bound of Degenerate Elliptical Vectors}\label{subsec:rex.geom}

In this section we consider the maximal magnitude of an elliptically distributed vector. A $p$-dimensional random vector $\bV$ is said to be elliptically distributed and is denoted as $\bV \sim \cE\cC_p (\bxi,\bTheta)$ if its density $f(\bv)$ satisfies that
\begin{equation}
  f(\bv) \propto g ((\bv-\bxi)^T \bTheta^{-1} (\bv-\bxi))
\end{equation}
for some continuous integrable function $g(\cdot)$ so that its isodensity contours are ellipses. The family of elliptical distributions is a generalization of multivariate Gaussian distributions and is an important and general class of distributions in practice \citep{fang1990symmetric}.

In this paper, we focus on an elliptical distributed vector $\bX \sim \cE\cC_p(\zero,\bSigma)$ with a covariance matrix $\bSigma$ that has unit diagonals. Through a packing argument, we find a functional link between the distribution of $\max_{1 \le j \le p} |X_j|$ and the rank of $\bSigma.$ we thus refer this link as the rank-extreme (ReX) association.

Below are the observations that connect these results to the packing problem: Consider any $p \times p$ covariance matrix $\bSigma$ that is positive semi-definite, has ones on the diagonal, and has rank $d$. Through its eigen-decomposition, we can write $\bSigma=\Lb^T\Lb$, where $\Lb=[\bL_1,\ldots,\bL_p]$ is a $d \times p$ matrix with columns $\bL_j$'s such that $\|\bL_j\|_2 =1$. Thus, we can write $\bX=\Lb^T\bZ$ where $\bZ \sim \cE\cC_d(\zero,\Ib).$ Moreover, for any $\bZ \sim \cE\cC_d(\zero,\Ib)$, if we consider the spherical coordinates, then we have $\bZ=\|\bZ\|_2 \bU$ where $\bU \sim Unif(\cS^{d-1})$. Note that $\|\bZ\|_2$ is a random variable which depends only on $d.$ We thus assume $\|\bZ\|_2$ is a random variable such that ${\|\bZ\|_2^2- u_d \over v_d} \stackrel{dist.}{\longrightarrow} F_\infty$ and ${\|\bZ\|_2^2 \over u_d} \stackrel{prob.}{\longrightarrow} 1$ where $u_d$ and $v_d$ are sequences of constants that depends only on $d,$ and $F_\infty$ is a proper random variable. Note also that $\|\bZ\|_2$ and $\bU$ are independent. Based on the above consideration, we obtain the following decomposition
\begin{equation}\label{eq.decom}
  \|\bX\|_\infty=\max_j |X_j|=\max_j |\langle \bL_j, \bZ \rangle|=\|\bZ\|_2  \max_j |\langle\bL_j, \bU \rangle|.
\end{equation}
Since the distribution of the maximal absolute inner products $ \max_j |\bL_j, \bU |$ is studied in Section~\ref{sec:spur}, we can apply these asymptotic results to study the distribution of $\|\bX\|_\infty=\max_j |X_j|.$ In particular, we develop the following universal bound on a degenerate elliptically distributed vector $\bX$ with a particular case of a degenerate Gaussian vector, where $\|\bZ\|_2^2 \sim \chi^2_d$ with $u_d=d$.

\begin{theorem}\label{thm.rex.arb}
\ \begin{enumerate}
  \item ({\bf ReX Bound for Degenerate Elliptical Vectors})\\
  For any vector of $p$ standard elliptical variables $\bX \sim \cE\cC_p(\zero,\bSigma)$ with $rank(\bSigma)=d$, the random variable $\|\bX\|_\infty=\max_j |X_j|$ satisfies that for any fixed $\delta>0,$
\begin{equation}\label{asymp.k.ell}
 \lim_{p,d \rightarrow \infty} \Pb\bigg(\|\bX\|_\infty  /\sqrt{u_d(1-p^{-2/(d-1)})} > 1+\delta\bigg) = 0.
\end{equation}
  \item ({\bf ReX Bound for Degenerate Gaussian Vectors})\\
  In particular, for any vector of $p$ standard Gaussian variables $\bX \sim \cN_p(\zero,\bSigma)$ with $rank(\bSigma)=d$, the random variable $\|\bX\|_\infty=\max_j |X_j|$ satisfies that for any fixed $\delta>0,$
\begin{equation}\label{asymp.k.gauss}
 \lim_{p,d \rightarrow \infty} \Pb\bigg(\|\bX\|_\infty  /\sqrt{d(1-p^{-2/(d-1)})} > 1+\delta\bigg) = 0.
\end{equation}
If further $d=d(p)$ with $\lim_{p \rightarrow \infty} (\log\log{p})^2d/(\log p)^2 \rightarrow \infty,$ then
\begin{equation}\label{asymp.k.highrank.gauss}
 \lim_{p \rightarrow \infty} \Pb\bigg(\|\bX\|_\infty /\sqrt{d(1-p^{-2/(d-1)})} \le 1\bigg) = 1.
\end{equation}
\end{enumerate}

\end{theorem}

Similar to the SABER $\sqrt{1-p^{-2/(n-1)}}$, this bound is universal over any correlation structures of rank $d.$ We also show that this bound is sharp, as described in Section~\ref{subsec:rex.iid}.

\subsection{Attainment of the ReX Bound and Related Limiting Distributions}\label{subsec:rex.iid}
The sharpness of the bound in Theorem~\ref{thm.rex.arb} was shown by considering the case when $\bL_j$'s in the decomposition \eqref{eq.decom} are i.i.d. uniformly distributed over $\cS^{d-1}$.
\begin{theorem}\label{thm.rex.iid} ({\bf Sharpness of ReX Bounds})
If $\bL_j$'s are i.i.d. uniformly distributed over the $(d-1)$-sphere $\cS^{d-1}, \forall j$ and are independent of $\bZ \sim \cE\cC_d(\zero,\Ib),$ then as $d\rightarrow \infty$ and $p \rightarrow \infty$,
\begin{equation}
      \max_{j} |\langle \bL_j, \bZ \rangle|/\sqrt{u_d(1-p^{-2/(d-1)})} \stackrel{prob.}{\longrightarrow} 1,
\end{equation}
i.e., $\forall \delta>0,$
\begin{equation}
 \lim_{p,d \rightarrow \infty} \Pb\bigg(\bigg|\max_{j} |\langle \bL_j, \bZ \rangle|  /\sqrt{u_d(1-p^{-2/(d-1)})} -1\bigg| > \delta\bigg) =0.
\end{equation}
In particular, if $\bZ \sim \cN_d(\zero,\Ib),$ then as $d\rightarrow \infty$ and $p \rightarrow \infty$,
\begin{equation}
      \max_{j} |\langle \bL_j, \bZ \rangle|/\sqrt{d(1-p^{-2/(d-1)})} \stackrel{prob.}{\longrightarrow} 1.
\end{equation}
\end{theorem}
One remark here is that though each realization of $\bL_j$'s results in a degenerate elliptically distributed $\bX$, unconditionally the joint distribution of $\bX$ is not elliptically distributed. Nevertheless, Theorem~\ref{thm.rex.iid} shows the existence of configurations of $\bL_j$ that attains the bound in Theorem~\ref{thm.rex.arb}.

The limit in Theorem~\ref{thm.rex.iid} indicates the following phase transition for the extreme value in degenerate Gaussian vectors, again depending on the limit of ${\log p \over d}$:
  \begin{enumerate}[(i)]
    \item If $d \rightarrow \infty$ and $\lim_{p \rightarrow \infty} \log{p}/d =\infty$, then
    $
    \max_{j} |\langle \bL_j, \bZ \rangle|/\sqrt{d} \stackrel{prob.}{\longrightarrow} 1.
    $
    \item If $\lim_{p \rightarrow \infty} \log{p}/d = \beta$ for fixed $0 <\beta < \infty$, then
    $
      \max_{j} |\langle \bL_j, \bZ \rangle|/\sqrt{\log{p}} \stackrel{prob.}{\longrightarrow} \sqrt{(1-e^{-2\beta})/\beta}.
    $
    \item If $\lim_{p \rightarrow \infty} \log{p}/d =0$, then
    $
      \max_{j} |\langle \bL_j, \bZ \rangle|/\sqrt{2 \log{p}}\stackrel{prob.}{\longrightarrow} 1.
    $
  \end{enumerate}
Note that the function $f(\beta)=(1-e^{-2\beta})/\beta$ is a smooth function for $\beta>0$ and its range is $(0,2).$ Thus, as the phase transition in Section~\ref{subsec:spur.iid}, the above phase transition is smooth. Moreover, the regime (iii) in the phase transition implies that when the rank $d$ is high compared to $\log p$, the maximum magnitude of a degenerate Gaussian vector can behave as that of i.i.d. Gaussian vectors.

Note that by \eqref{eq.decom}, we have the decomposition of the squared maximum norm
\begin{equation}\label{eq.rex.decom}
\|\bX\|_\infty^2=\max_{1 \le j \le p} |\langle \bL_j, \bZ \rangle|^2= \|\bZ\|_2^2 M_{p,d}^2.
\end{equation}
Thus, by the results in Section~\ref{subsec:spur.iid}, we also develop the following result on the limiting distribution of a degenerate elliptical vector when $\bL_j$'s are i.i.d. uniform.
\begin{theorem}\label{thm.rex.iid.lim}
\ \begin{enumerate}
  \item ({\bf Limiting Distribution of the Maximum of Degenerate Elliptical Vectors})\\
  Suppose $\bL_1,\ldots,\bL_p \stackrel{iid}{\sim} Unif(\cS^{d-1})$ and $\bZ \sim \cE\cC_d(\zero,\Ib)$ with ${\|\bZ\|_2^2- u_d \over v_d} \stackrel{dist.}{\longrightarrow} F_\infty$ for some sequences $u_d$, $v_d$ and a proper random variable $F_\infty$. Then with the constants $a_{p,d}$ and $b_{p,d}$ as in Theorem~\ref{thm.sabre.iid}, the random variable $K_{p,d}=\max_{1 \le j \le p} |\langle \bL_j, \bZ \rangle|^2=\|\bZ\|_2^2 M_{p,d}^2$ has the following limiting distribution:
\begin{enumerate}
  \item If $d$ is fixed and $p \rightarrow \infty,$ then
  $
    K_{p,d} \stackrel{dist.}{\longrightarrow} \|\bZ\|^2_2.
  $
  \item Suppose $d \rightarrow \infty$ and $p\rightarrow \infty.$
  \begin{enumerate}
      \item If $d \rightarrow \infty, p\rightarrow \infty,$ and ${v_d a_{p,d} \over u_d b_{p,d}} \rightarrow \infty,$ then
  $
    {K_{p,d}-u_d a_{p,d} \over v_d a_{p,d}} \stackrel{dist.}{\longrightarrow} F_\infty.
  $
  \item If $d \rightarrow \infty, p\rightarrow \infty,$ and ${v_d a_{p,d} \over u_d b_{p,d}} \rightarrow c $ with $0<c<\infty,$ then
  $
    {K_{p,d}-u_d a_{p,d} \over v_d a_{p,d}} \stackrel{dist.}{\longrightarrow} F_\infty+{1 \over c} H.
  $
  where $H \sim Gumbel(0,1),$ and $F_\infty$ and $H$ are independent.
  \item If $d \rightarrow \infty, p\rightarrow \infty,$ and ${v_d a_{p,d} \over u_d b_{p,d}} \rightarrow 0,$ then
  $
    {K_{p,d}-u_d a_{p,d} \over u_d b_{p,d}} \stackrel{dist.}{\longrightarrow} H
  $
  where $H \sim Gumbel(0,1)$.
  \end{enumerate}
\end{enumerate}

\item ({\bf Limiting Distribution of the Maximum of Degenerate Gaussian Vectors})\\
In particular, if $\bZ \sim \cN_d(\zero,\Ib),$ then the random variable $K_{p,d}$ has the following limiting distribution:
\begin{enumerate}
  \item If $d$ is fixed and $p \rightarrow \infty,$ then
  $
    K_{p,d} \stackrel{dist.}{\longrightarrow} \chi_d^2.
  $
  \item Suppose $d \rightarrow \infty$ and $p\rightarrow \infty.$
  \begin{enumerate}
      \item If $d \rightarrow \infty, p\rightarrow \infty,$ and $(\log{p})^2/d \rightarrow \infty,$ then
  $
    {K_{p,d}-d a_{p,d} \over \sqrt{2d} a_{p,d}} \stackrel{dist.}{\longrightarrow} G
  $
  where $G \sim \cN(0,1).$
  \item If $d \rightarrow \infty, p\rightarrow \infty,$ and $(\log{p})^2/d \rightarrow c $ with $0<c<\infty,$ then
  $
    {K_{p,d}-d a_{p,d} \over \sqrt{2d} a_{p,d}} \stackrel{dist.}{\longrightarrow} G+{1 \over \sqrt{2c}} H
  $
  where $G \sim \cN(0,1),$ $H \sim Gumbel(0,1),$ and $G$ and $H$ are independent.
  \item If $d \rightarrow \infty, p\rightarrow \infty,$ and $(\log{p})^2/d \rightarrow 0,$ then
  $
    {K_{p,d}-d a_{p,d} \over d b_{p,d}} \stackrel{dist.}{\longrightarrow} H
  $
  where $H \sim Gumbel(0,1)$.
  \end{enumerate}
\end{enumerate}
\end{enumerate}
\end{theorem}
Theorem~\ref{thm.rex.iid.lim} characterizes the limiting distribution of the squared maximum norm of degenerate elliptical vectors for the entire scope of the rank. The limiting distribution takes on a phase transition phenomenon according to the cross ratio between standardizing constants in the convergence of the norm and the convergence of the maximal squared inner product. This phenomenon is similar as the phase transitions in the classical extreme value theory for correlated random variables \citep{leadbetter1983,adler1990introduction,Haan2006}. When $\bZ$ is standard Gaussian distributed, the limiting distribution can be either $\chi_d^2$, standard Gaussian, a mixture of the standard Gaussian and Gumbel, or Gumbel depending on the relationship between $d$ and $p$.

\section{ReX Detection of Low-Dimensional Linear Dependency}\label{sec:rex.detect}
In this section we consider the problem of detection of low-rank dependency in high-dimensional Gaussian data. Suppose we have $n$ observations of a Gaussian vector $\bW \in \RR^p$ whose covariance matrix $\bSigma$ has rank is $rank(\bSigma)=d \ll p.$ One common technique in estimating $d$ is eigenvalue thresholding based on the principal component analysis (PCA). However, such methods become inaccurate when $n$ is small. Moreover, statistical inference, such as tests and confidence intervals, about $d$ as a parameter is not completely clear.

We propose to apply the rank-extreme association to obtain the information about $d$. We consider the following generating process of the data matrix $\Wb_{n \times p}$ from a factor model:
\begin{equation}
  \Wb_{n \times p} = \one_n \bmu^T+\Zb_{n \times d} \Lb_{d \times p} \Tb_{p \times p} +\sigma \Gb_{n \times p},
\end{equation}
where $\bmu$ is a fixed $p$-dimensional vector, $\Zb_{n \times d}$ has i.i.d. $\cN(0,1)$ entries, $\Lb_{d\times p}$ has columns of unit vectors, $\Tb_{p \times p}$ is a diagonal matrix with positive diagonal elements $\tau_1,\ldots, \tau_p,$ $\Gb_{n \times p}$ has i.i.d. $\cN(0,1)$ entries as the observation noises, and $\sigma \ge 0$ is the standard deviation of the noise. $\Zb$ and $\Gb$ are mutually independent so that each entry $W_{ij}$ is marginally distributed as $\cN(\mu_j,\tau_j^2 + \sigma^2).$ All of the above variables are not observed except for the data matrix $\Wb,$ and our goal is to estimate the rank $d$ with these observations.

Conventional estimate of $d$ is through a proper threshold over the eigenvalues of the sample covariance matrix of $\Wb$. Such an approach requires the eigenvalues to be at least $O(\sigma^2 \sqrt{p \over n})$ for possible detection, as shown in equation (7) and Theorem~1 in \citep{Kritchman200819}. In \citep{Kritchman200819}, the authors consider the case when $p=O(n)$ so that this required magnitude is $O(1).$ In general, to set this required magnitude to be $O(1)$ is equivalent to set $\sigma^2 = O(\sqrt{n/p}).$

In what follows, we introduce our ReX method for the inference of $d$ based on the observed extreme values. We consider both the case when the columns are i.i.d. uniform unit vectors and the general case.

\subsection{The Case When the Columns of $\Lb$ are i.i.d. Uniform Unit Vectors}\label{subsec:rex.detect.iid}
We first consider the case when the columns of $\Lb$ are realizations of i.i.d. uniform unit vectors over $\cS^{d-1}$. To explain our ReX method, we start with the elementary noiseless case when it is known that $\bmu=\zero$, $\sigma=0$, and $\tau_j=1$. In this case, we propose to approximate the asymptotic distribution of the maximal squared entry in each row of $\Wb$ by that of $K_{p,d}$. This approximation is particularly useful when $n \ll p,$ where obtaining the spectrum information is difficult from PCA based methods. The accuracy of the approximation is due to the following two reasons: (1) the theorems in Section~\ref{sec:rex} are for each row of $\Wb$ and have no requirement on $n$; (2) for each row, the condition $\sigma^2 = O(\sqrt{n/p})$ in turn shows that the largest magnitude of noise in each row of $\Wb$ is in the order of $O_p(\sqrt{2\log p}(n/p)^{1/4}).$ Thus, when $n/p \rightarrow 0$, this magnitude is $o_p(1)$ and will not affect the limiting distributions.

Note that for a large $p,$ Theorem~\ref{thm.sabre.iid}, the $\chi^2_d$ distribution, and the generalized extreme value distribution \citep{Haan2006} imply that
{\small
\begin{equation}
\begin{split}
   & \Eb [M_{p,d}^2] \\
  =& \Eb[\max_j |\langle\bL_j, \bU \rangle|^2]\\
             \sim & m_{p,d}:=a_{p,d} + {d-1 \over 2} (1-\Gamma(1+2/(d-1))) b_{p,d}\\
  & \Var [M_{p,d}^2] \\
  \sim & v_{p,d}:={(d-1)^2 b_{p,d}^2 \over 4}\big(\Gamma(1+4/(d-1))-\Gamma(1+2/(d-1))^2\big)
\end{split}
\end{equation}
}
where $a_{p,d}$ and $b_{p,d}$ are as in Theorem~\ref{thm.sabre.iid}. Thus, through \eqref{eq.rex.decom} and Theorem~\ref{thm.rex.iid}:
\begin{equation}\label{eq.rex.ev.approx}
\begin{split}
  \Eb [K_{p,d}] \sim &  E_{p,d}:=d m_{p,d},\\
  \Var [K_{p,d}] \sim & V_{p,d}:=2d (v_{p,d}+m_{p,d}^2) + d^2v_{p,d}.
\end{split}
\end{equation}
Suppose we observe $n$ i.i.d. samples of $K_{p,d}$ which are denoted as $K_{1,p,d},\ldots,K_{n,p,d}.$ By the central limit theorem we have
\begin{equation}
    \sqrt{n} {\bar{K}_{p,d} -E_{p,d} \over \sqrt{V_{p,d}}} \stackrel{dist.}{\longrightarrow} G
\end{equation}
where $\bar{K}_{p,d} = {1 \over n} \sum_{i=1}^n K_{i,p,d}$ and $G \sim \cN(0,1).$ An easy estimate of $d$ is thus the solution of the equation \begin{equation}\label{eq.rex.est}
\bar{K}_{p,d}=E_{p,d}.
\end{equation}

The estimators from this approach usually have a right-skewed distribution, as the distribution of $\chi_d^2$ and $M_{p,n}^2$ are both right-skewed. To reduce the right-skewness in the distribution of $K_{p,d},$ we take the square-root transformation and use the delta method as in \citep{Vaart1998} to obtain the following approximate probabilistic statement
\begin{equation}\label{eq.rex.ci}
\Pb\bigg(\bar{K}_{p,d} \ge \bigg( z_{\alpha} \sqrt{V_{p,d} \over 4nE_{p,d}} + \sqrt{E_{p,d}}\bigg)^2\bigg) \approx 1-\alpha
\end{equation}
where $0<\alpha<1$ and $z_{\alpha}$ is the $\alpha$-quantile of the standard Gaussian distribution. One then solves the inequality \begin{equation}\label{eq.rex.ci.solve}
\bar{K}_{p,d} \ge \bigg( z_{\alpha} \sqrt{V_{p,d} \over 4nE_{p,d}} + \sqrt{E_{p,d}}\bigg)^2
\end{equation}
in $d$ to obtain the $(1-\alpha)$-left-sided confidence interval from $0$ to this solution. Thus, probability statements about an unknown $d$ can be made. Note that $n$ needs not to be larger than $d$ throughout this approach.

Another advantage of the proposed inference method is the speed. Note that through the rank-extreme approach, there is no need of matrix multiplications. By quickly checking the maximal entry in each row, we may get a good sense of the rank as a parameter. Thus, much computation cost can be saved from the rank-extreme approach, and the proposed inference method for $d$ can be used for a fast detection of a low-rank.

When the parameters $\bmu$, $\sigma$, and $\tau_j$'s are unknown, we would need to estimate them. Since we are considering the case when $n$ is small while $p$ is large, the estimation of each component variance $\tau_j^2+\sigma^2$ is difficult. However, when it is known that $\tau_j$'s are equal to some unknown $\tau,$ we can estimate the variance $\tau^2+\sigma^2$ by borrowing strength from all variables. Specifically, we propose the following procedure for the inference of $d$:
\begin{enumerate}
  \item Center each column of $\Wb$ by subtracting the column averages. Denote the resulting data matrix by $\Wb_0.$
  \item Stack the columns of $\Wb_0$ into an $(np) \times 1$ vector and estimate the component standard deviation $\sqrt{\tau^2+\sigma^2}$ with this the sample standard deviation of this vector. Denote the estimate by $s.$
  \item Standardize $\Wb_0$ by dividing $s.$ Denote the resulting data matrix by $\Wb_s.$
  \item Apply the approach in the noiseless case above to $\Wb_s$ for inference about $d.$
\end{enumerate}
The above consideration is also applicable to the situation when the variables can be grouped into several blocks and the component variances within each block are close. Tests of equality variances such as \citep{Bartlett1937} are widely available. We will study the case with unequal variances in future work.

\subsection{General Case}\label{subsec:rex.detect.gen}
In this section we discuss the much more challenging situation when the columns of $\Lb$ are general unit vectors. For simplicity we restrict ourselves in the case when it is known that $\bmu=\zero$, $\sigma=0$, and $\tau_j=1$. We observe that by the decomposition \eqref{eq.decom}, we have the following proposition:
\begin{proposition}\label{cor.chi_inf}
Suppose $\bX \sim \cN_p(\zero,\bSigma)$ where $\bSigma$ has unit diagonals and $rank(\bSigma) =d.$ If there exists a collection of deterministic unit vectors $\bL_j$'s in $\RR^d$ such that $\bSigma= \Lb^T \Lb$ where $\Lb=[\bL_1,\ldots,\bL_p]$ and that for an independent uniformly distributed unit vector $\bU\in \RR^d$, $\max_j |\langle \bL_j,\bU \rangle| \stackrel{prob.}{\longrightarrow} 1$ as $p \rightarrow \infty$, then as $p \rightarrow \infty,$
\begin{equation}
  \max_j |X_j|^2 \stackrel{dist.}{\longrightarrow} \chi_d^2.
\end{equation}
\end{proposition}
With this proposition, we convert the inference about $d$ as a parameter to a simple inference problem on the degrees of freedom of a $\chi^2$ distribution. The condition $\max_j |\langle \bL_j,\bU \rangle| \stackrel{prob.}{\longrightarrow} 1$ is a condition on $\bSigma$ as $p \rightarrow \infty.$ It requires that the $p$ vectors $\bL_j$'s be ``densely'' distributed over the unit sphere in $\RR^d$ as $p$ increases, so that the minimal angle between the collection of $\bL_j$'s and the vector $\bU$ converges to 0 as the number of points on the unit sphere increases. The existence of such a $\bSigma$ is shown by the sharpness of the SABER. We aren't able to find a more precise condition on $\bSigma$ to guarantee the convergence as it relates to the challenging question of the optimal configuration of spherical cap packing and spherical code, on which some recent development includes \citep{cohn2014}. However, as long as $\lim_{p \rightarrow \infty} \Pb ( \max_j |\langle \bL_j,\bU \rangle | \ge 1-\delta) \ge 1-\varepsilon$ for some $\delta$ and $\varepsilon$, by conditioning on this event, inference such as confidence intervals can be made about $d$ as a parameter. Unfortunately, as many conditions in statistical literature, neither of these above conditions can be checked in practice. We will consider further analysis on this approach in future work.

\section{Simulation Studies}\label{sec:simu}
In this section we study the performance of the ReX detection of a low-rank from the model in Section~\ref{subsec:rex.detect.iid}. We consider two cases: (1) the case when it is known that $\bmu=\zero$, $\sigma=0$, and $\tau_j=1$ and (2) the case when the unknown component variances $\tau_j^2=\tau^2$ for some unknown $\tau$.

\subsection{Noiseless Case}\label{subsec:simu.noiseless}
In this subsection, we study the performance of the ReX detection when it is known that $\bmu=\zero$, $\sigma=0$, and $\tau_j=1$. We set $p=8000$, $n$ to be from \{10,20,30\}, and $d$ to be from $\{11,16,21\}$. In this case, the estimation of $d$ can be obtained by solving \eqref{eq.rex.est}, and the confidence interval can be obtained by solving \eqref{eq.rex.ci.solve}. We evaluate the performance of the ReX inference for $d$ with two criteria: (1) the sample mean squared error (MSE) of the point estimate of $d$ which is defined by
\begin{equation}
  MSE_{\hat{d}} = {1 \over N} \sum_{k=1}^N (d-\hat{d}_k)^2
\end{equation}
where $N$ is the number of simulations, and $\hat{d}_k$ is the estimate of $d$ from the $k$-th simulated data, $k=1,\ldots,N$; and (2) the coverage and $95\%$ upper bounds for $d$. As a comparison, we also study the MSE of an important PCA-based method, the KN method, proposed in \citep{Kritchman200819} by applying $\Wb$ to the algorithm posted on the authors' website.

Table~\ref{table.noiseless} represents simulation results on the performance of the ReX inference for different $n$'s and $d$'s. The results are based on 1000 simulated datasets. The first block in the table summarizes the MSE of the ReX estimation and the KN estimation. The second block shows the average coverage probability and the mean and median length of 95\% left-sided confidence intervals for $d$. When \eqref{eq.rex.ci.solve} does not have a solution, we record the confidence interval as not covering $d$.

In terms of estimation, although the MSE of the ReX estimation seems larger than that of the KN method in some cases, we noticed that in seven out of nine scenarios the KN method actually returns $n-2$ as an estimate of $d$. Indeed, the consistency of the KN method is shown when $n$ and $p$ are large and comparable, whereas its consistency is not guaranteed in these difficult situations when $p$ is much larger than $n$. In the scenarios in our simulations, the estimations of the KN method are not consistent and can lead to serious problems in practice, particularly when $n<d$. On the other hand, we see from Table~\ref{table.noiseless} that the MSE of the ReX estimation of $d$ gets better as $n$ grows. When the KN method returns better estimates, such as the cases when $n=30$ and $d=11$ or $d=16$, the ReX method has a much smaller MSE.

On the performance of ReX confidence intervals, note that the standard deviation of sample proportion of 1000 Bernoulli trials with success probability $0.95$ is about $0.007.$ Thus, a scenario with an average coverage between $0.936$ and $0.964$ shows a satisfactory confidence interval without being too liberal or too conservative. With this criterion, all ReX confidence intervals are satisfactory except when $n=10$ and $d=21$. In this case, not being able to solve \eqref{eq.rex.ci.solve} is the main reason of not covering $d$ in this difficult situation, see discussions at the end of this section. The length of the ReX confidence intervals is decreasing as $n$ increases. The median lengths are less than the mean lengths, showing the distribution of the upper bound of confidence intervals is indeed right-skewed, as expected in Section~\ref{subsec:rex.detect.iid}.

\begin{table*}[t]
\centering
\caption{Performance of ReX inference for different $n$'s and $d$'s when $p=8000$ in the unit variance and noiseless case. }\label{table.noiseless}
\ra{1.3}
\begin{tabular}{@{}rrrrcrrrcrrr@{}}\toprule
& \multicolumn{3}{c}{$d = 11$} & \phantom{abc}& \multicolumn{3}{c}{$d = 16$} &
\phantom{abc} & \multicolumn{3}{c}{$d = 21$}\\
\cmidrule{2-4} \cmidrule{6-8} \cmidrule{10-12}
& $n=10$ &  $n=20$ & $n=30$ && $n=10$ &  $n=20$ &  $n=30$ & & $n=10$ & $n=20$ & $n=30$ \\ \midrule
\multicolumn{3}{@{}l}{\em MSE of Estimation} \\
ReX & 12.40 & 3.97 &  2.91 && 38.07 & 12.69 & 8.73 && 73.52 & 47.21 & 22.03\\
KN & 9.00 & 49.00 & 148.16 && 64.00 & 4.00 & 143.98 && 169.00& 9.00 & 49.00\\
\multicolumn{3}{@{}l}{\em 95\% left-sided ReX Confidence Interval} \\
Coverage & 0.958 & 0.946 & 0.944 &&  0.951 & 0.949 & 0.950 && 0.926 & 0.937 & 0.952\\
Mean Upper Bound & 18.21 & 15.15 & 14.20 & & 31.41 & 23.66 & 22.13 & & 49.32 & 35.74 & 31.01\\
Median Upper Bound & 16.55 & 14.65 & 13.87 & & 26.27 & 22.34 & 21.48 & & 38.11 & 31.38 & 29.30  \\
\bottomrule
\end{tabular}
\end{table*}

\subsection{Equal Variance Case}\label{subsec:simu.equalv}
In this case, we set $p=8000$, $n$ to be from \{10,20,30\}, and $d$ to be from $\{11,16,21\}$ as in Section~\ref{subsec:simu.noiseless}. We set $\sigma$ to be $(n/p)^{1/4}$ as discussed in Section~\ref{sec:rex.detect}, set $\bmu$ to be a regular sequence of length $p$ from $-5$ to $5$, and set $\tau$ to be $2.$ Table~\ref{table.equalv} shows the results based on 1000 simulated datasets.

\begin{table*}[t]
\caption{Performance of ReX inference for different $n$'s and $d$'s when $p=8000$ in the equal variance with noise case. }\label{table.equalv}
\centering
\ra{1.3}
\begin{tabular}{@{}rrrrcrrrcrrr@{}}\toprule
& \multicolumn{3}{c}{$d = 11$} & \phantom{abc}& \multicolumn{3}{c}{$d = 16$} &
\phantom{abc} & \multicolumn{3}{c}{$d = 21$}\\
\cmidrule{2-4} \cmidrule{6-8} \cmidrule{10-12}
& $n=10$ &  $n=20$ & $n=30$ && $n=10$ &  $n=20$ &  $n=30$ & & $n=10$ & $n=20$ & $n=30$ \\ \midrule
\multicolumn{3}{@{}l}{\em MSE of Estimation} \\
ReX & 0.49 & 0.65 & 0.82 && 1.93  & 1.52 & 1.51 && 6.89 &  4.24 & 3.42\\
KN & 9 & 49 & 3.66 && 64 & 4 & 124.02  && 169 & 9 & 49.00\\
\multicolumn{3}{@{}l}{\em 95\%  left-sided ReX Confidence Interval} \\
Coverage & 1 & 1 & 1 &&  1 & 1 & 1 && 1 & 1 & 1\\
Mean Upper Bound  & 17.74 & 15.87 & 15.18 & & 28.10 & 24.27 & 22.84 & & 41.43 & 34.03 & 31.44 \\
Median Upper Bound  & 17.70 & 15.83 & 15.18 & & 27.77 & 24.21 & 22.76 & & 40.35 & 33.56 & 31.32  \\
\bottomrule
\end{tabular}
\end{table*}
On the estimation, Table~\ref{table.equalv} shows again the problem of PCA based methods when $p$ is much larger than $n$: the KN method returns $n-2$ for seven out of nine scenarios. When $n=30$ and $d=11$ or $d=16$, the KN method returns better estimates, but its MSE is larger than that of the ReX estimation. Note that in these two scenarios for the KN method as well as in all nine scenarios for the ReX method, the MSEs are much smaller than those in Table~\ref{table.noiseless}. One possible reason here is the standardization process. For the ReX method, recall from Section~\ref{subsec:rex.detect.iid} that the distribution of the estimators can be right-skewed. Since the variance estimation from the sample usually underestimates $\tau^2+\sigma^2$, the row maximum $K_{p,d}$ from standardized data can often be larger than that in the noiseless case, leading to a larger estimate of $d$ which offsets the right-skewness in the distribution.

On the ReX confidence intervals, Table~\ref{table.equalv} shows that the coverage probability of them is 1 for all nine scenarios. Although the coverage probability is conservative, the lengths of intervals are reasonably tight. Also, the median upper bounds are usually less than the mean ones, showing again the right-skewness. The problem of right-skewness is much more benign though.
\vskip .1in
In summary, in our simulation studies when $p$ is much larger than $n$, the traditional PCA based methods such as the KN method (1) may have a large MSE in estimating $d$, (2) may not be able to provide confidence intervals for $d$, and (3) requires matrix-wise calculation. On the other hand, the ReX inference (1) has a small MSE in estimation, (2) provides confidence interval statements for $d$, and (3) only needs to scan through the row maxima in the matrix and is thus fast. These results demonstrate the advantages of using the ReX method for the detection of a low-rank structure in high dimensions with a small sample size.

The simulation results also reflect some issues of the ReX method that need further improvements. For example, for some cases in Table~\ref{table.equalv}, the MSE of the ReX method increases as $n$ increases. This problem could be related to the approximation error in \eqref{eq.rex.ev.approx}. Also, the ReX inference are based on solutions of \eqref{eq.rex.est} and \eqref{eq.rex.ci.solve}. Such equations may not have a solution in difficult practical situations (This happens about 1\% of the time when $n=10$ and $d=21$). Although this problem seems to disappear when $n$ is above 10, a more stable algorithm is needed. We shall improve our method in these directions in future work.

\section{Discussions}\label{sec:disc}
We develop a probabilistic upper bound for the maximal inner product between any set of unit vectors and a stochastically independent uniformly distributed unit vector, as well as the limiting distributions of the maximal inner product when the set of unit vectors are i.i.d. uniformly distributed. We demonstrate the applications of these results the problems of spurious correlations and low-rank detections.

We emphasize that we focus our asymptotic theory in the uniform-$n$-large-$p$ paradigm. This type of asymptotics is motivated by the high-dimensional-low-sample-size framework \cite{Hall2005} which is emerging in many areas of science. The proposed packing approach can be especially useful in this framework because (1) finite-sample properties can be studied, and (2) existing packing literature can be applied. In the future, we will continue to explore this type of asymptotics in more general situations. For the theory, we plan to investigate the distribution of the maximal inner products with more generally correlated $\bL_j$'s. One of the applications of the new theory could be a more accurate detection method of a low rank. We also plan to improve and generalize the ReX detection method in the case when $\tau_j$'s are different, as well as in the case when the data are not Gaussian distributed.

\appendices

\section{Technical Lemmas}
We provide some key proofs in the appendix. Proofs of other results are immediate corollaries of these results. We start with the key observation that the distribution of each $|\langle \bL_j,\bU \rangle|^2$ is $Beta(1/2,(n-1)/2),$ as discussed at the beginning in Section~\ref{subsec:spur.uub} and also in \cite{muirhead1982}. Based on this fact, we first derive a lemma on the tail bounds of the $Beta(1/2,(n-1)/2)$ distribution. This lemma is proved by integration by parts, and the details are omitted.

\begin{lemma}\label{lemma.beta}
    For $0< w \le 1,$ we have the following bounds for an incomplete beta integral:
    \begin{equation}\label{ineq.beta}
    \begin{split}
     & {2((n+2)w-1) \over (n^2-1)} w^{-3/2}(1-w)^{n-1 \over 2}  \\
     \le &    \int_{w}^1 s^{-{1 \over 2}} (1-s)^{{n-3 \over 2}} ds \\
     \le & {2 \over (n-1)} w^{-1/2}(1-w)^{n-1 \over 2}
    \end{split}
    \end{equation}
\end{lemma}

We also find a lemma on the uniform convergence of the function $(n-1)(p^{2/(n-1)}-1).$ This lemma is important for the uniform convergence in the paper. The proof is easy and is omitted.
\begin{lemma}\label{lemma.uniform}
Uniformly for any $n \ge 2,$ as $p \rightarrow \infty,$ $(n-1)(p^{2/(n-1)}-1) \rightarrow \infty.$
\end{lemma}

We derive below a lemma summarizing the uniform convergence of standardizing constants in the theorems. Their proofs are routine analysis and are omitted.

\begin{lemma}\label{lemma.constants.beta}
Consider the sequences $a_{p,n} = 1- p^{-2/(n-1)} c_{p,n}$, $b_{p,n} = {2 \over n-1} p^{-2/(n-1)} c_{p,n}$ in Theorem~\ref{thm.sabre.iid} where $c_{p,n}=\big({n-1 \over 2}B\big({1 \over 2},{n-1 \over 2}\big)\sqrt{1-p^{-2/(n-1)}}\big)^{2/(n-1)}$ is a correction factor. For any fixed $x,$ let $w_{p,n}=a_{p,n}+b_{p,n}x.$ We have the following asymptotic results:
\begin{enumerate}
  \item Uniformly for any $n \ge 2$,  as $p \rightarrow \infty,$  $c_{p,n} / \big({n-1 \over 2}B\big({1 \over 2},{n-1 \over 2}\big)\big)^{2/(n-1)} \rightarrow 1,$  $b_{p,n} \rightarrow 0$, ${a_{p,n} \over 1-p^{-2/(n-1)}} \rightarrow 1,$ and ${b_{p,n} \over a_{p,n}} \rightarrow 0$.
  \item Uniformly for any $n \ge 2$, as $p \rightarrow \infty,$  ${(n+1)w_{p,n} \over (n+2)w_{p,n}-1} I\big(x \le {n-1 \over 2}\big) +I\big(x > {n-1 \over 2}\big) \rightarrow 1$.
\end{enumerate}
\end{lemma}

\section{Proofs in Section~\ref{sec:spur}}
\begin{proof}[Proof of Theorem~\ref{thm.saber}]

To show \eqref{ineq.sabre.finite}, note that for $\delta \ge 1/(p^{2/(n-1)}-1),$ $(1+\delta)(1-p^{-2/(n-1)}) \ge 1,$ thus the bound is trivial. Therefore, it is enough to show the convergence for any $\delta$ that $0<\delta < 1/(p^{2/(n-1)}-1)$. Similarly as the proof of Theorem~6.3 in \citep{berk2013a}, by Lemma~\ref{lemma.beta} and the inequalities that $\Gamma(x+1/2)/\Gamma(x) < \sqrt{x}$ as in \citep{jameson2013}, we have
{\small
\begin{equation}\label{eq.bonf}
\begin{split}
   &~~ \Pb\bigg(\, \max_{j} | \langle \bL_j, \bU \rangle | > \sqrt{(1+\delta)(1-p^{-2/(n-1)})}\bigg) \\
  ~~\le&~~ p \Pb \bigg(\, |\langle \bL_j, \bU \rangle| > \sqrt{(1+\delta)(1-p^{-2/(n-1)})}\bigg)      \\
 ~~\le & p \sqrt{2 \over (n-1) \pi} {(1- (1+\delta)(1-p^{-2/(n-1)}))^{(n-1)/2} \over \sqrt{(1+\delta)(1-p^{-2/(n-1)})}}  \\
 = & \sqrt{2 \over \pi (1+\delta)} {p^{1/(n-1)} \bigg(1-\delta(p^{2/(n-1)}-1)\bigg)^{(n-1)/2} \over \sqrt{(n-1)(p^{2/(n-1)}-1)}} \\
 \le & \sqrt{2 \over \pi (1+\delta)} {p^{1/(n-1)}\exp\bigg(-{1 \over 2}\delta(n-1)(p^{2/(n-1)}-1)\bigg) \over \sqrt{(n-1)(p^{2/(n-1)}-1)}} \\
\end{split}
\end{equation}
}
Thus, by Lemma~\ref{lemma.uniform},
\begin{equation}
  \Pb\bigg(\max_j |\langle \bL_j, \bU \rangle| > \sqrt{(1+\delta)(1-p^{-2/(n-1)})}\bigg) \rightarrow  0
\end{equation}
as $p \rightarrow \infty$ regardless of $n$.

To see \eqref{ineq.sabre.large}, note that if $\lim_{p \rightarrow \infty}n/\log{p}=\beta >0,$ then $p^{1/(n-1)} \rightarrow e^{1/\beta}<\infty.$ Thus we may set $\delta =0$ to get \eqref{ineq.sabre.large}. Also, if $n \rightarrow \infty$ but $n/\log{p} \rightarrow 0,$ then \eqref{eq.bonf} is further bounded by $\sqrt{2 \over \pi (n-1)}(1+o(1)).$ Thus we have \eqref{ineq.sabre.large}.
\end{proof}

\begin{proof}[Proof of Theorem~\ref{thm.sabre.sharp}]
Since we already have the upper bound, it is enough to show that for any fixed $\delta$ such that $0<\delta <1/2$,
\begin{equation}
  \Pb\bigg(\max_j |\langle \bL_j,\bU \rangle| < \sqrt{(1-\delta)(1-p^{-2/(n-1)})}\bigg) \rightarrow   0. \label{eq.1b}
\end{equation}
By the independence discussed at the beginning of Section~\ref{subsec:spur.iid}, we have that for $p \rightarrow \infty,$
{\small
\begin{equation}\label{eq.ind}
\begin{split}
   &~~ \Pb\bigg(\, \max_{j} | \langle \bL_j,\bU \rangle | < \sqrt{(1-\delta)(1-p^{-2/(n-1)})}\bigg) \\
  ~~=&~~ \bigg(\Pb \bigg(\, |\langle \bL_j,\bU \rangle| < \sqrt{(1-\delta)(1-p^{-2/(n-1)})}\bigg) \bigg)^p     \\
  ~~\le&~~ \exp\bigg(-p \Pb \bigg(\, |\langle \bL_j,\bU \rangle| > \sqrt{(1-\delta)(1-p^{-2/(n-1)})}\bigg) \bigg).
\end{split}
\end{equation}
}
We will lower-bound the absolute value of the exponent in \eqref{eq.ind}. By the lower bound in Lemma~\ref{lemma.beta} and the inequality that $\Gamma(x+1)/\Gamma(x+1/2) > \sqrt{x+1/4}$ as in \citep{jameson2013}, we have
{\small
\begin{equation}\label{eq.bound.kstar.lower}
  \begin{split}
    & p \Pb \bigg(\, |\langle \bL_j,\bU \rangle| > \sqrt{(1-\delta)(1-p^{-2/(n-1)})}\bigg)\\
    = & {p \over B\big(1/ 2,(n-1)/ 2\big)} \int_{(1-\delta) (1-p^{-2/(n-1)})}^1 \!\!\!\!\!\!\!\!\!\!\!\!\!\!\!x^{-1/2} (1-x)^{(n-3)/2} dx \\
  \ge & \sqrt{1 \over \pi (1-\delta)^3}{\sqrt{2n-3}\over n^2-1}{ p^{1/(n-1)} \over \sqrt{(p^{2/(n-1)}-1)^3}} \cdot \\
  & \big(1+\delta(p^{2/(n-1)}-1)\big)^{(n-1)/2}\big((1-\delta)n(p^{2/(n-1)}-1)-1\big) \\
  \ge & \sqrt{ 2 \over \pi(1-\delta)} {p^{1/(n-1)} \big(1+\delta(p^{2/(n-1)}-1)\big)^{(n-1)/2} \over \sqrt{n(p^{2/(n-1)}-1)}}  (1 +o(1)).
  \end{split}
\end{equation}
}
In the last step of \eqref{eq.bound.kstar.lower}, we use Lemma~\ref{lemma.uniform} again. It is now easy to see that
\begin{equation}
  p \Pb \bigg(\, |\langle \bL_j,\bU \rangle| > \sqrt{(1-\delta)(1-p^{-2/(n-1)})}\bigg) \rightarrow \infty
\end{equation}
as $p \rightarrow \infty$ regardless of the rate of $n=n(p)$, which completes the proof of Theorem~\ref{thm.sabre.sharp}.
\end{proof}

\begin{proof}[Proof of Theorem~\ref{thm.sabre.iid}]
If $x \ge (n-1)/ 2,$ then $a_{p,n} +b_{p,n} x \ge 1$ and the result is trivial. For $x< (n-1)/2,$ by Lemma~\ref{lemma.beta} and Lemma~\ref{lemma.constants.beta}, uniformly for any $n \ge 2,$ as $p \rightarrow \infty,$
{\small
\begin{equation}
  \begin{split}
       & -\log \Pb \bigg({M_{p,n}^2 - a_{p,n} \over b_{p,n}} < x\bigg) \\
    =& -\log \{ \Pb(|\langle \bL_j,\bU \rangle|^2 < b_{p,n} x + a_{p,n})^p\} \\
    =&  {2 p(1-a_{p,n} -b_{p,n}x)^{(n-1)/2} \over B\big({1 \over 2},{n-1 \over 2}\big)(n-1)\sqrt{a_{p,n} +b_{p,n} x} }(1+o(1))\\
    =& {2  p \big(1-1+c_{p,n} p^{-2 \over (n-1)}-{2 \over n-1}c_{p,n} p^{-2 \over (n-1)} x \big)^{(n-1) \over 2} \over B\big({1 \over 2},{n-1 \over 2}\big)(n-1)\sqrt{a_{p,n}(1+o(1))} }(1+o(1))  \\
    =& \big(1-{2 \over n-1} x\big)^{(n-1)/2} (1+o(1)).
  \end{split}
\end{equation}
}
When $n \rightarrow \infty, $ $\big(1-{2 \over n-1} x\big)^{(n-1)/2} \rightarrow e^{-x},$ which concludes the proof.
\end{proof}

\section{Proofs in Section~\ref{sec:rex}}
\begin{proof}[Proof of Theorem~\ref{thm.rex.arb}]
It is easy to show \eqref{asymp.k.ell} and \eqref{asymp.k.gauss}. To show \eqref{asymp.k.highrank.gauss}, note that for any $0< \varepsilon <1$,
{\small
\begin{equation}
\begin{split}
 & \Pb\bigg(\|\bX\|_\infty  /\sqrt{d(1-p^{-2/(d-1)})} > 1\bigg)\\
 =&  \Pb\bigg(\|\bZ \|_2 \max_j |\langle \bL_j,\bU \rangle|  /\sqrt{d(1-p^{-2/(d-1)})} > 1\bigg) \\
\le &  \Pb \bigg(\max_j |\langle \bL_j,\bU \rangle|  > \sqrt{(1-\varepsilon)(1-p^{-2/(d-1)})}\bigg)\\
& +\Pb\bigg(\|\bZ\|_2 > \sqrt{(1+\varepsilon)d} \bigg)
\end{split}
\end{equation}
}
We will show each of the two summands in the last line can be made small with a proper choice of $\varepsilon=\varepsilon(p).$

By the proof of Theorem~\ref{thm.saber}, we see that
\begin{equation}\label{asymp.k.bound1}
\begin{split}
 & \Pb \bigg(\max_j |\langle \bL_j,\bU \rangle| > \sqrt{(1-\varepsilon)(1-p^{-2/(d-1)})}\bigg)\\
\le & \sqrt{2 \over \pi (1-\varepsilon)} {p^{1/(d-1)}\exp\bigg({1 \over 2}\varepsilon(d-1)(p^{2/(d-1)}-1)\bigg) \over \sqrt{(d-1)(p^{2/(d-1)}-1)}}
\end{split}
\end{equation}
Note also that $\|\bZ\|_2^2 \sim \chi^2_d$. Thus by the Chernoff bound for $\chi^2_d$ distribution,
\begin{equation} \label{asymp.k.bound2}
\begin{split}
& \Pb\big(\|\bZ\|_2 > \sqrt{(1+\varepsilon)d} \big)\\
=& \Pb\big(\|\bZ\|_2^2 > (1+\varepsilon)d \big) \\
\le & ((1+\varepsilon)e^{-\varepsilon})^{d/2} \\
\le & e^{-d \varepsilon^2/6}
\end{split}
\end{equation}
Due to \eqref{asymp.k.bound1} and \eqref{asymp.k.bound2}, we let $\varepsilon=\varepsilon(p)={\log\log p /(4 \log p) }.$ In the case when $\lim_{p \rightarrow \infty} {(\log\log{p})^2d \over (\log p)^2} \rightarrow \infty,$ both \eqref{asymp.k.bound1} and \eqref{asymp.k.bound2} converge to $0$.

\end{proof}
\begin{proof}[Proof of Theorem~\ref{thm.rex.iid.lim}]
Note that,
\begin{equation}
  K_{p,d}- u_d a_{p,d} = M_{p,d}^2(\|\bZ\|_2^2- u_d) + u_d (M_{p,d}^2-a_{p,d})
\end{equation}
Now note also that $a_{p,d}$ is bounded and that $M_{p,d}/a_{p,d} \stackrel{prob.}{\longrightarrow} 1.$ Therefore, the theorem follows from Slutsky's theorem by checking the limit of the ratio $v_d a_{p,d}$ and  $u_d b_{p,d}$ and picking the one with a larger magnitude as the scaling factor.
\end{proof}

\section*{Acknowledgment}
The author appreciates the insightful suggestions from L.~D.~Brown, A.~Buja, T.~Cai, J.~Fan, J.~S.~Marron, and H.~Shen. The author thanks R.~Adler, J.~Berger, R.~Berk, A.~Budhiraja, E.~Candes, L.~de Haan, J.~Galambos, E.~George, S.~Gong, J.~Hannig, T.~Jiang, I.~Johnstone, A.~Krieger, R.~Leadbetter, R.~Li, D.~Lin, H.~Liu, J.~Liu, W.~Liu, Y.~Liu, Z.~Ma, X.~Meng, A.~Munk, A.~Nobel, E.~Pitkin, S.~Provan, A.~Rakhlin, D.~Small, R.~Song, J.~Xie, M.~Yuan, D.~Zeng, C.-H.~Zhang, N.~Zhang, L.~Zhao, Y.~Zhao, and Z.~Zhao for helpful discussions. The author also thanks the editors and reviewers for important comments that substantially improve the manuscript. The author is particularly grateful for L.~A.~Shepp for his inspiring introduction of the random packing literature.

This research is partially supported by NSF DMS-1309619, NSF DMS-1613112, NSF IIS-1633212, and the Junior Faculty Development Award at UNC Chapel Hill. This material was also partially based upon work supported by the NSF under Grant DMS-1127914 to the Statistical and Applied Mathematical Sciences Institute. Any opinions, findings, and conclusions or recommendations expressed in this material are those of the author(s) and do not necessarily reflect the views of the National Science Foundation.

\ifCLASSOPTIONcaptionsoff
  \newpage
\fi

\bibliographystyle{IEEEtran}
\bibliography{SCPARED}

\begin{IEEEbiographynophoto}{Kai Zhang}
received the Ph.D. degree in Mathematics from Temple University in 2007 and the Ph.D. degree in Statistics from the Wharton School, University of Pennsylvania, in 2012. He is now with the Department of Statistics and Operations Research at the University of North Carolina, Chapel Hill. His research interests include high-dimensional regression and inference, causal inference and observational studies, and quantum computing.
\end{IEEEbiographynophoto}

\end{document}